\documentclass[review]{amsart}
\usepackage[pagewise]{lineno}
\usepackage{url}
\usepackage{color}
\usepackage{algorithm}
\usepackage[margin=1.5in]{geometry}
\newtheorem{theorem}{Theorem}[section]
\newtheorem{lemma}[theorem]{Lemma}

\theoremstyle{definition}

\theoremstyle{assumptions}
\newtheorem{assumptions}{Assumption}
\theoremstyle{remark}
\newtheorem{remark}[theorem]{Remark}

\usepackage{float}
\usepackage{subfig}
\usepackage{tikz}
\usetikzlibrary{spy}
\usetikzlibrary{shapes.geometric}

\numberwithin{equation}{section}

\newcommand{\N}{\mathbb{N}}
\newcommand{\R}{\mathbb{R}}

\newcommand{\xk}{x_k}

\newcommand{\prox}{\mathrm{prox}}
\newcommand{\proj}{\mathrm{proj}}
\def\argmin{\operatornamewithlimits{argmin}\limits}

\begin{document}

\title[Plug-and-Play blind super-resolution of real MRI images]{Plug-and-Play blind super-resolution of real MRI images for  improved multiple sclerosis diagnosis}


\author[M. Cannas]{Matteo Cannas$^1$}\thanks{$^1$Dipartimento di Neuroscienze, Psicologia, Area del Farmaco e Salute del Bambino, Università di Firenze, Viale Pieraccini, 6, 50139 Firenze, Italy. \texttt{matteo.cannas@unifi.it}, \texttt{alice.mariottini@unifi.it}, \texttt{luca.massacesi@unifi.it}}
\author[A. Mariottini]{Alice Mariottini$^1$}
\author[L. Massacesi]{Luca Massacesi$^1$}
\author[F. Porta]{Federica Porta$^{2,*}$}\thanks{$^1$Dipartimento di Scienze Fisiche, Informatiche e Matematiche, 
Università di Modena e Reggio Emilia, Via Campi 213/b, 41125 Modena, Italy. \texttt{federica.porta@unimore.it}, \texttt{simone.rebegoldi@unimore.it}, \texttt{andrea.sebastiani@unimore.it}.}
\author[S. Rebegoldi]{Simone Rebegoldi$^2$}
\author[A. Sebastiani]{Andrea Sebastiani$^2$}

\thanks{$^*$Corresponding author}



\date{28/02/2026}

\dedicatory{}

\begin{abstract}
Magnetic resonance imaging (MRI) is central to the diagnosis of multiple sclerosis, where the identification of biomarkers such as the central vein sign benefits from high-resolution images. However, most clinical brain MRI scans are performed using 1.5 T scanners, which provide lower sensitivity compared to higher-field systems. We propose a blind super-resolution framework to enhance real 1.5 T MRI images acquired in clinical settings, where only post-processed data are available and the degradation model is not fully known. The problem is formulated as a non-convex blind inverse problem involving the joint estimation of the high-resolution image and the blur kernel. Image regularization is handled through a Plug-and-Play strategy based on a pretrained denoiser, while suitable constraints are imposed on the blur kernel. To solve the resulting model, we design a heterogeneous alternating block-coordinate method in which the two variables are updated using different types of algorithms. Convergence properties are rigorously established. Experiments on FLAIR and SWI sequences acquired at 1.5 T show improved structural definition and enhanced visibility of clinically relevant features, with visual comparison against 3 T images.
\end{abstract}
\maketitle

\section{Introduction}
Multiple sclerosis (MS) is the most frequent inflammatory demyelinating and degenerative disease of the central nervous system (CNS), with estimated 2.9 million affected people worldwide \cite{atlas-ms}. MS is a chronic disease with typical onset in young adulthood showing heterogeneous clinical manifestations such as motor, cerebellar and sensory symptoms \cite{atlas-ms}. MS still represents one of the main causes of non-traumatic disability in young adults, but its course can be affected by timely treatment with disease-modifying therapies \cite{hauser2020treatment}, providing that an accurate and early diagnosis is performed. Over the past two decades, magnetic resonance imaging (MRI) has become an indispensable tool for diagnosing MS, and the optional use of an MRI marker specific to MS, defined as the central vein sign (CVS), has been implemented in the latest revision of McDonald's diagnostic criteria \cite{montalban2025diagnosis}. Although the CVS can be assessed with different field strength MRI, the sensitivity is higher for high-field 7 T compared to 3T machines \cite{okromelidze2024central}, which are in turn more sensitive compared to 1.5 T MRI \cite{castellaro2020use}. However, in clinical practice, the vast majority of people receiving a brain MRI in non-tertiary centres are scanned at magnetic field strengths of 1.5 T due to several reasons including, but not limited to, higher costs (nominal cost of 1 M dollar for Tesla of magnetic fields), increased specific absorption rate (SAR) and acoustic noise, more intense artifacts from metallic objects such as surgical clips, more common field heterogeneities and susceptibility artifacts, lower safety of some materials used in medical operations \cite{sarracanie2015low, thomas2019pros}. Improving the diagnostic usability of widely available 1.5 T MRI systems without requiring hardware upgrades would therefore have a significant impact on routine clinical practice.

To support physicians in achieving accurate diagnoses, the main aim of this work is to propose a super-resolution approach designed to enhance the overall quality of 1.5 T MRI images acquired from patients with multiple sclerosis. From a mathematical perspective, the problem can be cast as a blind super-resolution task. Indeed, since the data are acquired in a real clinical setting, only limited information about the imaging acquisition process is available; consequently, the underlying blurring degradation is not fully known.  In particular, the acquisition process can be formalized as
\begin{equation}\label{eq:model}
b = S(\theta*x) + \eta,
\end{equation}
where $b\in\mathbb{R}^m$  denotes the observed low-resolution image, $*$ represents the convolution operator between the unknown high-resolution image $x\in\mathbb{R}^n$  and the  unknown blur kernel $\theta\in\mathbb{R}^p$, $S\in\mathbb{R}^{m\times n}$ is a standard $s$-fold downsampling matrix such that $n=s^2\times m$, and $\eta\in\mathbb{R}^m$ models additive white Gaussian noise with standard deviation $\sigma\in\mathbb{R}^+$. The computation of the convolution product is performed by means of a matrix-vector multiplication $\theta*x = \Theta x = X \theta$, where $X$ and $\Theta$ are suitable structured matrices depending on the choice of the boundary conditions \cite{Hansen-etal-2006}.

The inverse problem associated with \eqref{eq:model} is severely ill-posed due to the non-uniqueness of its solution. According to the Bayesian paradigm, an approximate solution $(x,\theta)$ can be found by solving the following generally non-convex optimization problem     
\begin{equation} \label{eq:opt_probl_general}
\underset{x\in\mathbb{R}^n,\theta\in\mathbb{R}^p}{\operatorname{argmin}} \ \frac{1}{2}\|S(\theta*x) - b\|^2 +  \phi(x) + \psi(\theta),   
\end{equation}
where, $\phi:\mathbb{R}^n\rightarrow\mathbb{R}$ and $\psi:\mathbb{R}^p\rightarrow\mathbb{R}$ are proper regularization terms for $x$ and $\theta$, respectively. 

The optimization problem in \eqref{eq:opt_probl_general} is a block-variable composite problem with separable, possibly non smooth, regularization terms acting on each variable. From a numerical point of view, problems of this kind are typically addressed using alternating optimization schemes, in which the objective function is minimized sequentially with respect to one variable at a time in a Gauss-Seidel fashion \cite{Grippo-etal-1999,Grippo-etal-2000}. Among these approaches, the class of alternating forward-backward (FB) algorithms has been extensively studied in the recent literature \cite{Attouch-etal-2010,Attouch-etal-2013,Bolte-etal-2014,Frankel-etal-2015,Grippo-etal-2000,Razaviyayn-etal-2013,Tseng-etal-2009} and has also been successfully applied to imaging problems  \cite{Abboud-etal-2017,Chouzenoux-etal-2016,Cornelio-etal-2015}. Denoted by $f(x,\theta)$ the data fidelity term $\frac{1}{2}\|S(\theta*x)-b\|^2$, classical alternating FB methods for addressing \eqref{eq:opt_probl_general} can be written as
 \begin{equation}\label{eq:alt_FB}
     \left\{
     \begin{array}{l}
     x_{k+1} = \prox_{\alpha_x \phi}\left(x_k - \alpha_x\nabla_{x}f(x_k,\theta_k)\right)\\
     \\
     \theta_{k+1} = \prox_{\alpha_\theta \psi}\left(\theta_k-\alpha_\theta\nabla_{\theta}f(x_{k+1},\theta_k)\right)
     \end{array}
     \right.,
 \end{equation}
where $\alpha_x$ and $\alpha_{\theta}$ are two proper positive steplengths and $\prox_{\alpha_x\phi}, \prox_{\alpha_\theta\psi}$ denote the proximal operator of $\alpha_x\phi(\cdot)$ and $\alpha_\theta\psi$, respectively. In particular, given $\alpha\in\mathbb{R}^+$ and a proper, lower semicontinuous function $h:\mathbb{R}^s\rightarrow\mathbb{R}\cup\{\infty\}$, the proximal operator of $\alpha h$ is defined as
$$
\prox_{\alpha h}(\bar{z}) = \underset{z\in\mathbb{R}^s}{\operatorname{argmin}}\ \frac{1}{2}\|z-\bar{z}\|^2 + \alpha h(z), \quad \forall \bar{z}\in\mathbb{R}^s.
$$

\textit{Contributions.} To develop an effective approach for super-resolving real 1.5 T MRI images, it is necessary to both properly select the regularization terms $\phi$ and $\psi$ in \eqref{eq:opt_probl_general} and design an efficient alternating FB method to solve the resulting optimization problem. The contributions of this paper address both aspects. Specifically:

\begin{itemize}
\item[\textit{(i)}] For image regularization we adopt a Plug-and-Play (PnP) framework, leveraging a pretrained denoising neural network. In addition, physically meaningful constraints on the blur kernel are incorporated to prevent trivial and non-physical solutions.
\item[\textit{(ii)}] We propose an alternating FB algorithm tailored to the considered resulting optimization model. Compared with the scheme in \eqref{eq:alt_FB}, we adopt a more general strategy. Specifically, the two coordinate blocks are not both updated via a standard proximal gradient scheme, where the proximal step is associated with the regularization term and the gradient step with the data fidelity term. Instead, each block is updated using a method specifically designed for its corresponding regularization. The data fidelity term $f(\cdot,\cdot)$ is not used symmetrically for the computation of the gradients with respect to the two variables, as in \eqref{eq:alt_FB}; rather, it plays distinct roles in the update of each block. {Convergence results for the proposed algorithm are rigorously established.}
\end{itemize}
{Numerical experiments on the super-resolution of  images acquired using a 1.5 T MRI scanner (Siemens) at the Neuroradiology Unit of Careggi University Hospital in Florence were carried out. Two different MRI sequences were considered, namely FLAIR (Fluid-Attenuated Inversion Recovery) and SWI (Susceptibility-Weighted Imaging). FLAIR images enable visualization of white-matter lesions, whereas SWI images highlight venous structures within brain tissue. The results obtained with the proposed approach appear promising and potentially useful for improving the clinical use of 1.5 T MRI images for specific purposes, such as the identification of the central vein sign (CVS), a diagnostic imaging marker. In particular, white-matter lesions and venous structures can be identified with improved definition and sharper edges.}
As a gold standard for visual comparison, FLAIR and SWI MRI images acquired on a 3T machine (Ingenia, Philips) from the same patient at the Neuroradiology Unit of the Hospital were adopted.
 
\textit{Related works.} State-of-the-art approaches to MRI super-resolution are predominantly based on deep learning methods (see, e.g., \cite{Dong-etal-2014,Du-etal-2020,Jia-etal-2017,Jia-etal-2016,Zhang-etal-2022,Zhao-etal-2021,Zhao-etal-2019} and references therein). Despite their strong performance, deep super-resolution methods have some limitations. First, the large number of parameters results in substantial memory and computational demands. Moreover, in supervised approaches, performance is often constrained by the limited availability of high-quality annotated datasets, especially in the biomedical context. Self-supervised approaches \cite{Zhao-etal-2021,Zhao-etal-2019}, while avoiding the need for external paired training data by exploiting intrinsic structural information within each individual image, typically require training or fine-tuning the model separately for each new image. This requirement increases computational cost and processing time and may limit generalization across different acquisition settings.

\noindent In parallel, several model-based methods for MRI super-resolution have been proposed in the literature (see, e.g., \cite{Bano-etal-2020,Beirinckx-etal-2022,Lajous-etal-2020,Poot-etal-2010,VanSteenkiste-etal-2017}). However, these approaches generally rely on classical regularization techniques, such as Total Variation or Tikhonov regularization. In contrast to PnP frameworks, they do not incorporate learned denoisers as implicit priors. Furthermore, limited attention has been devoted to the design and analysis of optimization schemes specifically tailored to the resulting models.

\noindent Moreover, we remark that, in our application, only post-processed MRI images are available, whereas the raw acquisition data are not accessible. For this reason, we adopt a general blind super-resolution framework, as formulated in \eqref{eq:opt_probl_general}, which can in principle be extended to other imaging acquisition systems. Within this setting, the main goal of this work is to address real MRI image super-resolution by integrating PnP regularization within the blind  problem formulation \eqref{eq:opt_probl_general} and by designing an ad hoc optimization scheme with provable convergence guarantees. While a few recent works address block-coordinate methods for blind imaging problems with PnP regularization \cite{Gan-etal-2023,Huang-etal-2025,Sun-etal-2020}, they typically rely on the same class of numerical solvers for all subproblems. In contrast, we propose a heterogeneous optimization scheme, where the two blocks are updated through algorithms of different nature, resulting in a more flexible framework.


\textit{Contents.} The paper is organized as follows. Section \ref{sec:2} introduces the complete problem formulation considered, including the regularization terms. The proposed algorithm and its convergence analysis are presented in Section \ref{sec:3}. Numerical experiments are reported in Section \ref{sec:4}. Section \ref{sec:5} is devoted to concluding remarks and future perspectives.

\section{The proposed model}\label{sec:2}
 The appropriate choice of $\phi$ and $\psi$ in the blind super-resolution formulation \eqref{eq:opt_probl_general} is crucial to obtain accurate reconstructions.  In this paper, for the regularization of the unknown image $x$, we adopt a Plug-and-Play (PnP) based approach. The PnP framework represents a powerful and innovative paradigm for solving
inverse problems in computational imaging, enabling the incorporation of learned priors into the problem formulation via the regularization term \cite{hurault2022gradient,Kamilov-etal-2023,Venkatakrishnan-etal-2013}. Over the past decade, PnP methods have achieved state-of-the-art performance across a broad range of image restoration tasks \cite{Ahmad-etal-2020,Buzzard-etal-2018,Meinhardt-etal-2017,Yuan-etal-2020,Zhang-et-al-2021}. Motivated by these results and according to the approach proposed in \cite{hurault2022gradient}, we consider the following definition for $\phi$ in \eqref{eq:opt_probl_general}:
\begin{equation}\label{eq:phi}
    \phi(x) = \frac{\lambda}{2}\|x-N_{\sigma}(x)\|^2,
\end{equation}
where $\lambda\in\mathbb{R}^+$ is a positive regularization parameter and $N_{\sigma}:\mathbb{R}^n\rightarrow\mathbb{R}^n$ is any differentiable neural network architecture that has proven to be efficient for image denoising, given a noise level $\sigma\in\mathbb{R}^+$. Specifically, we fix the function $N_\sigma$ as the UNet model proposed in \cite{Zhang-et-al-2021} and later trained within a PnP framework as described in \cite{hurault2022gradient}. We employ the pretrained network weights released by the authors of \cite{hurault2022gradient}, which are available through the DeepInverse library \cite{tachella2025deepinverse}. Additional implementation details are provided in Section \ref{sec:4}. Notably, our approach is fully unsupervised, as it relies exclusively on a pretrained model without requiring additional task-specific training. 

On the other hand, by following standard approaches in the literature,  we impose classical physical constraints on the kernel $\theta$, namely non-negativity of the components and flux normalization. Moreover, as successfully done in \cite{Benfenati-etal-2023,Bertero-etal-2018,Prato-etal-2013},  we incorporate prior knowledge of the Strehl ratio (SR) of the optical system to further restrict the solution space of \eqref{eq:opt_probl_general}.  The SR is defined as the ratio between the peak intensity of an aberrated kernel and that of an ideal, diffraction-limited kernel. This prior information naturally induces an upper bound on the kernel values. The inclusion of this additional constraint effectively rules out trivial and non-physical solutions, such as the Dirac delta kernel, which would otherwise satisfy the remaining constraints. As a consequence, the regularization term $\psi$ in \eqref{eq:opt_probl_general} is chosen as the indicator function of the following closed and convex set
\begin{equation}\label{eq:Omega}
    \Omega = \left\{\theta\in\mathbb{R}^p \ : \ 0\leq\theta\leq M \ \mbox{and} \ \sum_{i} \theta_i = 1\right\},
\end{equation}
where $M$ is an approximation of the SR, which can be properly estimated from an ideal kernel (see for example \cite{Ashida-etal-2020,Chen-etal-2017}).

\section{The proposed approach}\label{sec:3}
In this section we detail a possible approach to address the blind super-resolution optimization problem \eqref{eq:opt_probl_general}, combined with the definition of $\phi$ and $\psi$ from Section \ref{sec:2}. For convenience, we report the resulting problem below. Specifically, we are interested in solving
\begin{equation}\label{eq:opt_prob_blind_super_res}
\underset{x\in\mathbb{R}^n,\theta\in\mathbb{R}^p}{\operatorname{argmin}} \ F(x,\theta)\equiv \frac{1}{2}\|S(\theta * x)-b\|^2 + \frac{\lambda}{2}\|x-N_{\sigma}(x)\|^2  + \iota_{\Omega}(\theta),
\end{equation}
where
\begin{equation*}
\iota_{\Omega}(\theta)=\begin{cases}
0, \quad &\text{if }\theta\in\Omega\\
+\infty, \quad &\text{otherwise}
\end{cases}
\end{equation*}
denotes the indicator function of the set $\Omega$ defined in \eqref{eq:Omega}. Hereafter, the data fidelity term will be denoted by $f(x,\theta)$. 

The method we propose to solve problem \eqref{eq:opt_prob_blind_super_res} is an alternating FB based algorithm, summarized in Algorithm \ref{alg:1}. 
\begin{algorithm}
\raggedright
\vspace{0.1cm}
Choose $x_{-1}, {x}_0\in\R^n$, $\theta_0\in\Omega$, $\alpha_x>0$, $\alpha_{\theta}>0$, $\rho\geq 0$ and $\nu,\gamma \in(0,1)$. \vspace{0.1cm}\\
{\sc For} $k=0,1,\ldots$
\begin{enumerate}
    \item[] {\hspace{-5mm}\sc Step 1.} Update of $x_{k+1}$.\vspace{0.1cm}\\
    \begin{enumerate}
    \item[] {\hspace{0.25cm}\sc Step 1.1.} Set $y_k = x_k + \rho\left(\nabla \phi(x_{k-1}) - \nabla \phi(x_k)\right)$.\vspace{0.1cm}\\
    \item[] {\hspace{0.25cm}\sc Step 1.2.} Set $x_{k+1} =  \prox_{\alpha_x f(\cdot,\theta_k)}(y_k - \alpha_x \nabla \phi(x_k))$.
    \end{enumerate}\vspace{0.2cm}
    \item[] {\hspace{-5mm}\sc Step 2.} Update of $\theta_{k+1}$.\vspace{0.1cm}\\
    \begin{enumerate}
    \item[] {\hspace{0.25cm}\sc Step 2.1.} Compute $\hat{\theta}_{k}=\proj_{\Omega}(\theta_k-\alpha_{\theta}\nabla_{\theta} f(x_{k+1},\theta_k))$. \vspace{0.1cm}\\
    \item[] {\hspace{0.25cm}\sc Step 2.2.} Set $d_k = \hat{\theta}_{k} - \theta_k$ and $\lambda_k = 1$. \vspace{0.1cm}\\
    {\hspace{2cm}\sc If} $f(x_{k+1},\theta_k+\lambda_k d_k) > f (x_{k+1},\theta_k) + \nu \lambda_k \nabla_{\theta} f(x_{k+1},\theta_k)^Td_k$\vspace{0.1cm}\\
    \hspace{2.5cm} Set $\lambda_k = \gamma\lambda_k$ and go to {\sc Step 2.2};\vspace{0.1cm}\\
    {\hspace{2cm}\sc Else} \vspace{0.1cm}\\
    \hspace{2.5cm} go to {\sc Step 2.3}.\vspace{0.1cm}
    \item[] {\hspace{0.25cm}\sc Step 2.3.} {\sc If} $f(x_{k+1},\hat{\theta}_k)<f(x_{k+1},\theta_k+\lambda_k d_k)$
    \vspace{0.1cm}\\
    {\hspace{2.6cm}Set} $\theta_{k+1}=\hat{\theta}_k$\\
    \vspace{0.1cm}
    {\hspace{2cm}\sc Else}
    \vspace{0.1cm}\\
    {\hspace{2.6cm}Set} $\theta_{k+1} = \theta_k + \lambda_k d_k.$
    \end{enumerate}
\end{enumerate}
\caption{Asymmetric alternating forward-backward method}
\label{alg:1}
\end{algorithm}
Unlike the standard alternating FB scheme \eqref{eq:alt_FB}, in which the fidelity term is used to perform the forward step in the updates of both $x_{k+1}$ and $\theta_{k+1}$, Algorithm \ref{alg:1} employs the fidelity term to perform the backward step when updating $x_{k+1}$ and the forward step when updating $\theta_{k+1}$. This asymmetry reflects the different analytical structures of the two subproblems and allows the use of effective yet different algorithms for the two updates. In more detail, the update of $x_{k+1}$ is carried out via one step of the so called forward-reflected-backward method \cite{Malitsky-etal-2020,Tseng2000,Wang-etal-2022}, whereas the update of $\theta_{k+1}$ consists of one iteration of a proximal gradient method with line search \cite{Bonettini-Loris-Porta-Prato-2016,Bonettini-etal-2009}. We explain below the motivations behind adopting these two different strategies.
\begin{description}
    \item[$x$-update] When minimizing the objective function in \eqref{eq:opt_prob_blind_super_res} with respect to $x$ while keeping $\theta$ fixed, the problem reduces to a super-resolution problem with a PnP regularization term. Following \cite{hurault2022gradient}, this problem can be efficiently solved by a FB algorithm of the form
    \begin{equation}\label{eq:PnpGD}
    x_{k+1} = \prox_{\alpha_xf(\cdot,\theta_k)}(x_k-\alpha_x\nabla \phi(x_k)),
    \end{equation}
    where $\alpha_x$ is a suitable positive steplength parameter. The proximal operator of $\alpha_x f(\cdot,\theta_k)$ can be computed efficiently in closed form. Although both $f(\cdot,\theta_k)$ and $\phi$ are smooth, an FB approach is preferable to a full gradient scheme for numerical and practical reasons. In particular, proximal steps may improve asymptotic performance \cite{Combettes-etal-2019}, the Lipschitz constant of $\nabla \phi + \nabla f$ can be much larger than that of $\nabla f$ \cite[Section~9.3]{Beck-2017}, leading to smaller admissible steplengths in a full gradient method, and the FB strategy combined with PnP regularization terms has shown strong empirical results in imaging applications \cite{hurault2022gradient}. In \cite{hurault2022gradient}, the authors additionally suggest applying a further \textit{denoising} step to the last FB iteration in order to enhance the quality of the reconstruction. Indeed it can be shown that $\phi$ in \eqref{eq:phi} induces a denoising operator $D_\sigma$ as
    $$
    D_\sigma(x) = x - \nabla \phi(x).
    $$
    As a consequence, in \cite{hurault2022gradient} the final reconstruction is defined as $D_\sigma(x_K)$, where $x_K$ is the last iterate generated by \eqref{eq:PnpGD}. To incorporate this heuristic within our alternating FB scheme while preserving convergence guarantees, we modify \eqref{eq:PnpGD} by including the reflected \textsc{Step 1.1}. This step can be interpreted as a generalized denoising operation and allows the $x$-update to fall within a class of methods for which convergence properties are established and inherited by our scheme. 
    \item[$\theta$-update] The regularization term on the unknown kernel naturally suggests the use of a gradient projection method to solve the subproblem associated with $\theta$. Specifically, we consider the algorithm proposed in \cite{Bonettini-etal-2017}. The projection onto the  set $\Omega$ defined in \eqref{eq:Omega} can be formulated as a root-finding problem and efficiently computed via the secant-like algorithm proposed in \cite{Dai-etal-2006}. The line search strategy on the additional parameter $\lambda_k\in(0,1]$ ensures a sufficient decrease of the objective function restricted to the $\theta$-block. This also avoids potentially demanding fine-tuning of the steplength parameter $\alpha_\theta$. {Since $f$ is a  continuously differentiable function with respect to $\theta$, the line search procedure in \textsc{Step 2.2} of Algorithm \ref{alg:1} is well defined \cite[Theorem 2.4]{Brigin-etal-2000}}.
\end{description}
The proposed asymmetric alternating FB method can be regarded as nonstandard in the existing literature and belongs to a research line that has not yet been thoroughly explored. This claim is supported by the following considerations.
While several standard block-coordinate FB methods for possibly nonconvex optimization have been proposed (see, for example, \cite{Bonettini-Prato-Rebegoldi-2018,Chouzenoux-etal-2016,Frankel-etal-2015,Gur-Sabach-Shtern-2023,Ochs-2019,Xu-Yin-2013,Xu-Yin-2017} and references therein), only \cite{Gur-etal-2022} considers an alternating minimization strategy in which each subproblem is solved using a different algorithm. 
Moreover, only a limited number of works address block-coordinate algorithms specifically tailored to PnP regularization terms \cite{Gan-etal-2023,Huang-etal-2025,Sun-etal-2020}. However, the approaches proposed in these works employ the same type of method to solve the subproblems associated with the block updates.
Finally, we are not aware of any block-coordinate methods involving forward-reflected-backward schemes.
\subsection{Convergence properties}
In this section, we present the convergence analysis of Algorithm \ref{alg:1} applied to a general optimization problem of the form 
\begin{equation}\label{eq:abstract_problem}
\underset{x\in\mathbb{R}^n,\theta\in\mathbb{R}^p}{\operatorname{argmin}} \ F(x,\theta)\equiv f(x,\theta) + \phi(x)  + \iota_{\Omega}(\theta),
\end{equation}
where $\Omega$ is a nonempty closed and convex subset of $\mathbb{R}^p$ and the functions $f$ and $\phi$ satisfy the following assumptions.
\begin{assumptions}\label{ass:1}
        \begin{itemize}
        \item[\textit{(i)}] The function $f:\mathbb{R}^n\times\mathbb{R}^p\rightarrow\mathbb{R}$ is continuously differentiable 
        with respect to all its arguments and $\nabla f$ is locally Lipschitz continuous, i.e., for all bounded subsets $K_1\times K_2\subseteq \R^n\times \R^p$ there exists $L(K_1,K_2)>0$ such that \begin{align}\label{eq:nablaf_Lipschitz}
\|\nabla f(x_1,\theta_1)&-\nabla f(x_2,\theta_2)\|\nonumber\\
&\leq L(K_1,K_2)\|(x_1-x_2,\theta_1-\theta_2)\|, \quad \forall \ x_1,x_2\in K_1, \ \forall \ \theta_1,\theta_2\in K_2. 
\end{align}
Furthermore, for any $x\in\R^n$, $\theta\in\R^p$, the partial gradients $\nabla_x f(\cdot,\theta)$ and $\nabla_\theta f(x,\cdot)$ are Lipschitz continuous, i.e., there exist $L_x(\theta)>0$ and $L_\theta(x)>0$ such that
        \begin{equation}\begin{aligned}\label{eq:block_Lipschitz}
            \|\nabla_{x} f(x_1,\theta)-\nabla_{x} f(x_2,\theta)\|&\leq L_x(\theta)\|x_1-x_2\|, \quad \forall \ x_1,x_2\in\R^n\\
            \|\nabla_{\theta} f(x,\theta_1)-\nabla_{\theta} f(x,\theta_2)\|&\leq L_\theta(x)\|\theta_1-\theta_2\|, \quad \forall \ \theta_1,\theta_2\in\R^p.
        \end{aligned}
        \end{equation}
        In addition, for each bounded subset $K_1\times K_2\subseteq \R^n\times \R^p$, there exist $L(K_1)>0$ and $L(K_2)>0$ such that
        \begin{equation}\label{eq:Lbound}
        \begin{aligned}
         \sup\{L_\theta(x): \ x\in K_1\}&\leq L(K_1)\\  
         \sup\{L_x(\theta): \ \theta\in K_2\}&\leq L(K_2).
        \end{aligned}   
        \end{equation}
        Finally, for every fixed $\theta$, the proximal operator of $f(\cdot,\theta)$ is single-valued.
        \item[\textit{(ii)}] The function $\phi:\mathbb{R}^n\rightarrow\mathbb{R}$ is continuously differentiable and $\nabla \phi$ is Lipschitz continuous, i.e., there exists $L_{\phi}>0$ such that  
\begin{equation*}\label{eq:phi_Lipschitz}
\|\nabla \phi(x)-\nabla \phi(y)\|\leq L_{\phi}\|x-y\|, \quad \forall \ x,y\in \mathbb{R}^n. 
\end{equation*}
        \item[\textit{(iii)}] The function $f+\phi$ is bounded from below.
    \end{itemize}
\end{assumptions}
\begin{remark}
The functions $f(x,\theta) = \frac{1}{2}\|S(\theta*x)-b\|^2$ and $\phi(x) = \frac{\lambda}{2}\|x-N_\sigma(x)\|^2$ in \eqref{eq:opt_prob_blind_super_res} satisfy Assumption \ref{ass:1}. Indeed $f$ is continuously differentiable on
$\mathbb{R}^n\times\mathbb{R}^p$, $\nabla f$ is locally Lipschitz continuous and its partial gradients are uniformly Lipschitz continuous. Moreover, 
for every fixed $\theta$, 
$f(\cdot,\theta)$ is convex. Therefore, the corresponding proximal subproblem
is strongly convex and admits a unique minimizer. Consequently, the proximal operator of $f(\cdot,\theta)$ is single-valued. {The assumption on $\phi$ holds whenever the neural network $N_\sigma$ is given by the composition of continuously differentiable functions with bounded and Lipschitz continuous derivatives \cite[Proposition 2]{hurault2022gradient}. This condition is satisfied by the convolutional neural network employed in the numerical experiments presented in Section \ref{sec:4}}. 
\end{remark}

In our analysis, we are interested in proving convergence to a stationary point of problem \eqref{eq:abstract_problem}. More precisely, let us consider the limiting subdifferential of the objective function $F$ at any point $(x,\theta)\in\R^n\times \R^p$ \cite[Definition 2.3]{Rockafellar-Wets-1998}, which by the block-coordinate structure of \eqref{eq:abstract_problem} and standard calculus rules \cite[Lemma 3]{Bonettini-Prato-Rebegoldi-2018} can be written as
\begin{equation}\label{eq:subdifferential_F}
\partial F(x,\theta)=\left\{\left(\begin{array}{c}
     \nabla_x f(x,\theta)+\nabla \phi(x)  \\
     \nabla_\theta f(x,\theta)+ w
\end{array}\right): \ w\in\partial \iota_\Omega(\theta)\right\},    
\end{equation}
where $\partial \iota_{\Omega}(\theta)=\{v\in\R^p: \ v^T(\theta'-\theta)\leq 0, \ \forall \theta'\in\R^p\}$ is the normal cone of $\Omega$ at $\theta$. Then, a stationary point $(x_*,\theta_*)\in\R^n\times\R^p$ for $F$ is defined by the following inclusion:
\begin{equation}\label{eq:stationary}
    0\in\partial F(x_*,\theta_*).
\end{equation}

We begin the convergence analysis by introducing a merit function $H:\R^n\times \R^n\times \R^p\rightarrow \R\cup\{+\infty\}$ defined as
\begin{align}\label{eq:merit}
H(x,y,\theta)&=F(x,\theta)+\frac{1}{4\alpha_x}\|x-y\|^2\nonumber\\
&=f(x,\theta)+\phi(x)+\iota_{\Omega}(\theta)+\frac{1}{4\alpha_x}\|x-y\|^2, \quad \forall \ (x,y,\theta)\in\R^n\times\R^n\times \R^p.
\end{align}
The merit function in \eqref{eq:merit} extends to the block-coordinate setting the one considered in \cite{Bot-etal-2016,Wang-etal-2022} for the convergence analysis of single-block forward-reflected-backward methods. Note that the subdifferential of the function $H$ writes as
\begin{equation}\label{eq:merit_subdifferential}
\partial H(x,y,\theta) = \left\{\left(
\begin{array}{c}
\nabla_x f(x,\theta)+\frac{1}{2\alpha_x}(x-y)+\nabla \phi(x)\\
\frac{1}{2\alpha_x}(y-x)\\
\nabla_\theta f(x,\theta)+w
\end{array}
\right): \ w\in \partial \iota_{\Omega}(\theta)\right\},
\end{equation}
and comparing \eqref{eq:merit_subdifferential} with \eqref{eq:subdifferential_F}, it is easy to see that if $(x_*,y_*,\theta_*)$ is stationary for $H$, then $(x_*,\theta_*)$ is stationary for $F$. 

The following lemma establishes a decreasing property for the merit function \eqref{eq:merit}. Indeed, although Algorithm \ref{alg:1} does not ensure the decrease of the objective function $F$ at each iteration (because of the reflected point $y_k$ in {\sc Step 1.1}), we can still prove that the merit function \eqref{eq:merit} decreases along the iterates. The first part of the proof is an adaptation of the arguments employed to prove \cite[Lemma 3.2]{Wang-etal-2022}. However, unlike in \cite{Wang-etal-2022}, we must take into account the block-coordinate structure of Algorithm \ref{alg:1} and the presence of a line search procedure in {\sc Step 2.2}.
\begin{lemma}\label{lem:H1}
Let Assumption \ref{ass:1} hold. Suppose that the sequence $\{(\xk,\theta_k)\}_{k\in\N}$ generated by Algorithm \ref{alg:1} is bounded. Define the sequence $\{z_k\}_{k\in\N}\subseteq \R^n\times \R^n$ as
\begin{equation*}\label{eq:zk}
z_k = \left(\begin{array}{c}
     x_{k+1}\\
     x_{k}
\end{array}\right), \quad \forall \ k\in\N,   
\end{equation*}
as well as the sequence $\{\bar{z}_k\}_{k\in\N} \subseteq \R^n\times \R^n\times \R^p$ given by
\begin{equation}\label{eq:barzk}
\bar{z}_k=\left(\begin{array}{c}
     x_{k}\\
     x_{k-1}\\
     \theta_k
\end{array}\right)= \left(\begin{array}{c}
     z_{k-1}\\
     \theta_k
\end{array}\right), \quad \forall \ k\in\N.
\end{equation}
If $\rho<\frac{1}{2L_\phi}$ and $\alpha_x<\frac{1-2L_\phi\rho}{2L_\phi}$, then  there exists $c>0$ such that
\begin{equation}\label{eq:sufficient_decrease}
H(\bar{z}_{k+1})\leq H(\bar{z}_{k}) - c\|\bar{z}_{k+1}-\bar{z}_k\|^2, \quad \forall \ k\in \N.
\end{equation}
\end{lemma}
\begin{proof}
From \textsc{Step 1.2} of Algorithm \ref{alg:1} and the definition of proximal operator of $\alpha_xf(\cdot,\theta_k)$, it holds that
\begin{equation}\label{eq:x_k+1}
\begin{aligned}
    x_{k+1} &= \argmin_{x\in\mathbb{R}^n} \frac{1}{2} \|x - (y_k-\alpha_x\nabla \phi(x_k))\|^2+\alpha_x f(x,\theta_k)\\
    &=\argmin_{x\in\mathbb{R}^n} \frac{1}{2\alpha_x}\|x-y_k\|^2+\nabla \phi(x_k)^T(x-y_k) + f(x,\theta_k).
    \end{aligned}
\end{equation}
As a consequence of \eqref{eq:x_k+1}, it follows that
\begin{equation*}
\begin{aligned}
    \frac{1}{2\alpha_x}\|x_{k+1}-y_k\|^2+\nabla \phi(x_k)^T(x_{k+1}-y_k) + f(x_{k+1},\theta_k)&\leq \frac{1}{2\alpha_x}\|x_{k}-y_k\|^2+ \\
    &\quad +\nabla \phi(x_k)^T(x_{k}-y_k) + f(x_{k},\theta_k)
    \end{aligned}
\end{equation*}
and, hence, 
\begin{equation}\label{eq:ineq1_f}
    \begin{aligned}
        f(x_{k+1},\theta_k)&\leq f(x_{k},\theta_k)+\nabla \phi(x_k)^T(x_{k}-x_{k+1}) + \frac{1}{2\alpha_x}\|x_{k}-y_k\|^2 \\
        &\quad - \frac{1}{2\alpha_x}\|x_{k+1}-y_k\|^2.
    \end{aligned}
\end{equation}
    In view of Assumption \ref{ass:1}{(ii)}, we can apply the descent lemma \cite[Theorem 18.15]{Bauschke-etal-2017} to $\phi$, which yields the inequality
    \begin{equation}\label{eq:ineq1_phi}
        \phi(x_{k+1})\leq \phi(x_k) + \nabla \phi(x_k)^T(x_{k+1}-x_k) + \frac{L_\phi}{2}\|x_{k+1}-x_k\|^2.
    \end{equation}
    Combining \eqref{eq:ineq1_f} and \eqref{eq:ineq1_phi} leads to
    \begin{equation*}
        \begin{aligned}
            f(x_{k+1},\theta_k)+\phi(x_{k+1}) &\leq
            f(x_k,\theta_k) + \phi(x_k) + \frac{L_\phi}{2}\|x_{k+1}-x_k\|^2 + \frac{1}{2\alpha_x}\|x_k-y_k\|^2\\
            &\quad -\frac{1}{2\alpha_x}\|x_{k+1}-y_k\|^2\\
            &= f(x_k,\theta_k) + \phi(x_k) + \frac{L_\phi}{2}\|x_{k+1}-x_k\|^2 -\frac{1}{2\alpha_x}\|x_k-x_{k+1}\|^2\\
            &\quad - \frac{1}{\alpha_x}(x_{k+1}-x_k)^T(x_k-y_k)\\
            &=f(x_k,\theta_k) + \phi(x_k) + \left(\frac{L_\phi}{2}-\frac{1}{2\alpha_x}\right)\|x_k-x_{k+1}\|^2 \\
            &\quad + \frac{1}{\alpha_k}(x_k-x_{k+1})^T(x_k-y_k)\\
            &\leq f(x_k,\theta_k) + \phi(x_k) + \left(\frac{L_\phi}{2}-\frac{1}{2\alpha_x}\right)\|x_k-x_{k+1}\|^2 \\
            &\quad + \frac{1}{\alpha_k}\|x_k-x_{k+1}\|\|x_k-y_k\|\\
            &= f(x_k,\theta_k) + \phi(x_k) + \left(\frac{L_\phi}{2}-\frac{1}{2\alpha_x}\right)\|x_k-x_{k+1}\|^2 \\
            &\quad +\frac{\rho}{\alpha_x}\|x_k-x_{k+1}\|\|\nabla \phi(x_{k-1})-\nabla \phi(x_k)\|\\
            &\leq f(x_k,\theta_k) + \phi(x_k) + \left(\frac{L_\phi(\alpha_x+\rho)}{2\alpha_x}-\frac{1}{2\alpha_x}\right)\|x_k-x_{k+1}\|^2 \\
            &\quad + \frac{L_\phi \rho}{2\alpha_x}\|x_k-x_{k-1}\|^2\\
            &=f(x_k,\theta_k) + \phi(x_k) + \left(\frac{L_\phi(\alpha_x+\rho)}{2\alpha_x}-\frac{1}{2\alpha_x}\right)\|z_k-z_{k-1}\|^2\\
            &\quad -\left(\frac{L_\phi(\alpha_x+\rho)}{2\alpha_x}-\frac{1}{2\alpha_x}\right)\|x_k-x_{k-1}\|^2+ \frac{L_\phi \rho}{2\alpha_x}\|x_k-x_{k-1}\|^2,
        \end{aligned}
    \end{equation*}
    where the first equality follows from the identity $\|x_k-y_k\|^2 = \|x_{k+1}-y_k\|^2-\|x_{k+1}-x_k\|^2 - 2(x_{k+1}-x_k)^T(x_k-y_k)$, the second inequality is a consequence of the Cauchy-Schwarz inequality, the third equality holds in view of \textsc{Step 1.1} of Algorithm \ref{alg:1}. Moreover the third inequality follows by combining Assumption \ref{ass:1}{(ii)} and the inequality $2ab\leq a^2+b^2$ and the fourth equality holds by recalling the definition of $z_k$ and the norm identity 
\begin{equation}\label{eq:norm_id}\|z_{k}-z_{k-1}\|^2 = \|x_{k+1}-x_k\|^2+\|x_k-x_{k-1}\|^2.\end{equation}
By neglecting the negative term $-\frac{L_\phi}{2}\|x_k-x_{k-1}\|^2$ in the right-hand side of the previous inequality, we can conclude that
    \begin{equation}\label{eq:ineq_f+phi}
        \begin{aligned}
      f(x_{k+1},\theta_k)+\phi(x_{k+1}) &\leq  f(x_k,\theta_k) + \phi(x_k) + \left(\frac{L_\phi(\alpha_x+\rho)}{2\alpha_x}-\frac{1}{2\alpha_x}\right)\|z_k-z_{k-1}\|^2 \\
            &\quad +\frac{1}{2\alpha_x}\|x_k-x_{k-1}\|^2.   
        \end{aligned}
    \end{equation}
    We are now able to establish a decreasing property of the merit function associated with \textsc{Step 1} of Algorithm \ref{alg:1}. Specifically,
    \begin{equation*}
        \begin{aligned}
            &H(z_k,\theta_k) = H(x_{k+1},x_k,\theta_k)  = f(x_{k+1},\theta_k)+\phi(x_{k+1})+\frac{1}{4\alpha_x}\|x_{k+1}-x_k\|^2\\
            &\leq f(x_k,\theta_k)+\phi(x_k)+ \left(\frac{L_\phi(\alpha_x+\rho)}{2\alpha_x}-\frac{1}{2\alpha_x}\right)\|z_k-z_{k-1}\|^2 +\frac{1}{2\alpha_x}\|x_k-x_{k-1}\|^2\\
            &\quad +\frac{1}{4\alpha_x}\|x_{k+1}-x_k\|^2\\
            &= f(x_k,\theta_k)+\phi(x_k) +\frac{1}{4\alpha_x}\|x_k-x_{k-1}\|^2 + \left(\frac{L_\phi(\alpha_x+\rho)}{2\alpha_x}-\frac{1}{4\alpha_x}\right)\|z_k-z_{k-1}\|^2\\
            &=H(z_{k-1},\theta_k) + \left(\frac{L_\phi(\alpha_x+\rho)}{2\alpha_x}-\frac{1}{4\alpha_x}\right)\|z_k-z_{k-1}\|^2,
        \end{aligned}
    \end{equation*}
    where the first inequality is a direct consequence of \eqref{eq:ineq_f+phi} and the last equality follows from the definition of the merit function in \eqref{eq:merit}. In view of the assumptions on $\alpha_x$ and $\rho$, it is easy to see that $\frac{L_\phi(\alpha_x+\rho)}{2\alpha_x}-\frac{1}{4\alpha_x}<0$ and, hence, we can conclude that there exist $c_x>0$ such that
    \begin{equation}\label{eq:dec_step1}
    H(z_k,\theta_k)\leq H(z_{k-1},\theta_k) - c_x\|z_k-z_{k-1}\|^2.
    \end{equation}
We now analyze \textsc{Step 2} of Algorithm \ref{alg:1}. By reasoning similarly to \eqref{eq:x_k+1}, we can write that
\begin{equation}\label{eq:min_theta}
\hat{\theta}_k = \argmin_{\theta\in\mathbb{R}^p} \nabla_\theta f(x_{k+1},\theta_k)^T(\theta-\theta_k) + \frac{1}{2\alpha_\theta}\|\theta-\theta_k\|^2+\iota_{\Omega}(\theta).
\end{equation}
Consequently,
$$
\nabla_\theta f(x_{k+1},\theta_k)^T(\hat{\theta}_k-\theta_k)+\frac{1}{2\alpha_\theta}\|\hat{\theta}_k-\theta_k\|^2\leq 0,
$$ 
and in view of the definition of $d_k$ in \textsc{Step 2.2}, we obtain
\begin{equation}\label{eq:d_k}
\nabla_\theta f(x_{k+1},\theta_k)^Td_k\leq -\frac{1}{2\alpha_\theta}\|d_k\|^2.
\end{equation}

Furthermore, let $K\subseteq \R^n$ be any bounded subset containing $\{x_k\}_{k\in\N}$. Since $\nabla_\theta f(x_{k+1},\cdot)$ is $L(K)-$Lipschitz continuous for all $k\in\N$, thanks to \eqref{eq:block_Lipschitz}-\eqref{eq:Lbound}, it is possible to prove that there exists $\lambda_{\min}>0$ dependent on $L(K)$ such that \cite[Proposition 3.2]{Bonettini-Loris-Porta-Prato-2016}
\begin{equation}\label{eq:lambdamin}
\lambda_k\geq \lambda_{\min}, \quad \forall \ k\in\N.
\end{equation}

Hence, by the line search inequality in \textsc{Step 2.2}, we observe that
\begin{equation}\label{eq:f_theta_k+1}
    \begin{aligned}
        f(x_{k+1},\theta_{k+1})&\leq f(x_{k+1},\theta_k)+\nu\lambda_k\nabla_\theta f(x_{k+1},\theta_k)^Td_k\\
        &\leq f(x_{k+1},\theta_k)-\frac{\nu\lambda_k}{2\alpha_\theta}\|d_k\|^2\\
        &\leq  f(x_{k+1},\theta_k)-\frac{\nu\lambda_k}{2\alpha_\theta}\|\theta_{k+1}-\theta_k\|^2\\
        &\leq f(x_{k+1},\theta_k)-\frac{\nu\lambda_{\min}}{2\alpha_\theta}\|\theta_{k+1}-\theta_k\|^2,
    \end{aligned}
\end{equation}
where the second inequality follows from \eqref{eq:d_k}, the third inequality from the property $\|\theta_{k+1}-\theta_k\|\leq \|d_k\|$, which is implied by \textsc{Step 2.3} and $\lambda_k\leq 1$, and the last inequality from \eqref{eq:lambdamin}. A decreasing property of the merit function associated with \textsc{Step 2} of Algorithm \ref{alg:1} can be obtained at this point. Indeed,
\begin{equation}\label{eq:dec_step2}
    \begin{aligned}
    H(z_k,\theta_{k+1}) &= H(x_{k+1},x_k,\theta_{k+1}) = f(x_{k+1},\theta_{k+1})+\phi(x_{k+1}) + \frac{1}{4\alpha_x}\|x_{k+1}-x_k\|^2\\
    &\leq f(x_{k+1},\theta_k)+\phi(x_{k+1})+\frac{1}{4\alpha_x}\|x_{k+1}-x_k\|^2 -\frac{\nu\lambda_{\min}}{2\alpha_\theta}\|\theta_{k+1}-\theta_k\|^2\\  
    &=H(z_k,\theta_k) - \frac{\nu\lambda_{\min}}{2\alpha_\theta}\|\theta_{k+1}-\theta_k\|^2,
    \end{aligned}
\end{equation}
 where the inequality follows from \eqref{eq:f_theta_k+1}. 
Denoting by $c_\theta$ the positive constant $\frac{\nu\lambda_{\min}}{2\alpha_\theta}$ and plugging \eqref{eq:dec_step1} into \eqref{eq:dec_step2}, we get
$$
H(z_k,\theta_{k+1})\leq H(z_{k-1},\theta_k)-c_x\|z_k-z_{k-1}\|^2 - c_\theta \|\theta_{k+1}-\theta_k\|^2.
$$
This last inequality combined with the definition of $\bar{z}_k$ and the norm identity 
\begin{equation*}
\begin{aligned}
\|\bar{z}_{k+1}-\bar{z}_k\|^2 &= \|(z_k,\theta_{k+1})-(z_{k-1},\theta_k)\|^2=\|(z_k-z_{k-1},\theta_{k+1}-\theta_k)\|^2\\
&=\|z_k-z_{k-1}\|^2+\|\theta_{k+1}-\theta_k\|^2    
\end{aligned}
\end{equation*}
yields
$$
H(\bar{z}_{k+1})\leq H(\bar{z}_k)-c\|\bar{z}_{k+1}-\bar{z}_k\|^2,
$$
where $c=\min\{c_x,c_\theta\}$.
\end{proof}

The next Lemma states that the norm of a subgradient $v_{k+1}$ computed at the iterate $\hat{z}_{k+1}=(x_{k+1}^T,x_{k}^T,\hat{\theta}_k^T)^T$ is upper bounded by a quantity proportional to the gap $\|\bar{z}_{k+1}-\bar{z}_k\|$ between two successive iterates generated by Algorithm \ref{alg:1}. A similar result was proved for a single-block forward-reflected-backward method in \cite[Lemma 3.5]{Wang-etal-2022}.

\begin{lemma}\label{lem:H3}
Let Assumption \ref{ass:1} hold. Suppose that the sequence $\{(\xk,\theta_k)\}_{k\in\N}$ generated by Algorithm \ref{alg:1} is bounded. Define the sequence $\{\hat{z}_k\}_{k\in\N}$ as
\begin{equation}\label{eq:tildezk}
\hat{z}_k=\left(\begin{array}{c}
     x_{k}\\
     x_{k-1}\\
     \hat{\theta}_{k-1}
\end{array}\right)
, \quad \forall \ k\in\N.
\end{equation}
If $\rho < \frac{1}{2L_\phi}$ and $\alpha_x<\frac{1-2L_\phi \rho}{2L_\phi}$, then there exists $q>0$ and a sequence of vectors $\{v_{k+1}\}_{k\in\N}$ such that $v_{k+1}\in\partial H(\hat{z}_{k+1})$ and
\begin{equation*}
\|v_{k+1}\|\leq q\|\bar{z}_{k+1}-\bar{z}_k\|, \quad \forall \ k\in\N.
\end{equation*}
\end{lemma}

\begin{proof}
Before proceeding with the proof, we observe that the sequence of the projected-gradient points $\{\hat{\theta}_k\}_{k\in\N}$ is bounded, thanks to the boundedness assumption on $\{(x_k,\theta_k)\}_{k\in\N}$, {\sc Step 2.1} of Algorithm \ref{alg:1} and the continuity of the projected-gradient operator with respect to the variables $x$ and $\theta$. Then, we let $K_1\times K_2\subseteq \R^n\times \R^p$ be any bounded subset such that
\begin{equation}\label{eq:K}
\{(x_k,\theta_k)\}_{k\in\N}\cup  \{(x_k,\hat{\theta}_k)\}_{k\in\N}\subseteq K_1\times K_2. 
\end{equation}
This is crucial as we want to apply Assumption \ref{ass:1}(i)) on a bounded subset to which the iterates of interest belong to.

Based on the optimality condition associated to the minimum problem \eqref{eq:min_theta} for computing the iterate $\hat{\theta}_{k}$, there exists $w_{k+1}\in \partial \iota_{\Omega}(\hat{\theta}_k)$ such that
\begin{equation*}
\nabla_\theta f(x_{k+1},{{\theta}_k})
+\frac{1}{\alpha_\theta}(\hat{\theta}_k-\theta_k)+w_{k+1}=0,
\end{equation*}
or equivalently
\begin{equation}\label{eq:wk}
w_{k+1}=
-\nabla_\theta f(x_{k+1},{{\theta}_k})
+\frac{1}{\alpha_\theta}(\theta_k-\hat{\theta}_k)\in\partial \iota_{\Omega}(\hat{\theta}_k).
\end{equation}
Furthermore, for all $k\in\N$, we define the vector
\begin{equation}
v_{k+1}=\left(
\begin{array}{c}
\nabla_x f(x_{k+1},\hat{\theta}_k)+\frac{1}{2\alpha_x}(x_{k+1}-x_k)+\nabla \phi(x_{k+1})\\
\frac{1}{2\alpha_x}(x_k-x_{k+1})\\
\nabla_\theta f(x_{k+1},\hat{\theta}_k)
+w_{k+1}
\end{array}
\right).
\end{equation}
By writing explicitly the subdifferential \eqref{eq:merit_subdifferential} at the point $\hat{z}_{k+1}=(x_{k+1},x_k,\hat{\theta}_k)$ and recalling from \eqref{eq:wk} that $w_{k+1}\in\partial \iota_{\Omega}(\hat{\theta}_k)$, it follows that $v_{k+1}\in\partial H(\hat{z}_{k+1})$. Then, the norm of the subgradient $v_{k+1}$ is upper bounded as follows
\begin{align}\label{eq:tech1}
\|v_{k+1}\|&\leq \left\|\nabla_x f(x_{k+1},\hat{\theta}_k)+\frac{1}{2\alpha_x}(x_{k+1}-x_k)+\nabla \phi(x_{k+1})\right\|\nonumber\\
&+\frac{1}{2\alpha_x}\|x_k-x_{k+1}\|+\|\nabla_\theta f(x_{k+1},\hat{\theta}_k)+w_{k+1}\|\nonumber\\
&=\left\|\nabla_x f(x_{k+1},\hat{\theta}_k)+\frac{1}{2\alpha_x}(x_{k+1}-x_k)+\nabla \phi(x_{k+1})\right\|\nonumber\\
&+\frac{1}{2\alpha_x}\|x_k-x_{k+1}\|+
{\left\|\nabla_\theta f(x_{k+1},\hat{\theta}_k) - \nabla_\theta f(x_{k+1},\theta_k) + \frac{1}{\alpha_\theta}(\theta_k-\hat{\theta}_k)\right\|}\nonumber\\
&\leq \left\|\nabla_x f(x_{k+1},\hat{\theta}_k)+\frac{1}{2\alpha_x}(x_{k+1}-x_k)+\nabla \phi(x_{k+1})\right\|\nonumber\\
&+\frac{1}{2\alpha_x}\|x_k-x_{k+1}\|+
{\left(L(K_1)+\frac{1}{\alpha_\theta}\right)\|\hat{\theta}_k-\theta_k\|},
\end{align}
where the equality is obtained by replacing $w_{k+1}$ with \eqref{eq:wk} {and the second inequality follows from the triangular inequality and Assumption \ref{ass:1}(i), together with the definition of the upper bound $L(K_2)$ in \eqref{eq:Lbound}, with $K_2$ given in \eqref{eq:K}}. By summing and subtracting $\nabla_x f(x_{k+1},\theta_k)$ in the first norm of the right-hand side of \eqref{eq:tech1}, and applying the triangular inequality, we get
\begin{align}\label{eq:tech2}
\|v_{k+1}\|&\leq  \left\|\nabla_x f(x_{k+1},\hat{\theta}_k)- \nabla_x f(x_{k+1},\theta_k)\right\|\nonumber\\
&+\left\|\nabla_x f(x_{k+1},\theta_k)+\frac{1}{2\alpha_x}(x_{k+1}-x_k)+\nabla \phi(x_{k+1})\right\|\nonumber\\
&+\frac{1}{2\alpha_x}\|x_k-x_{k+1}\|+{\left(L(K_1) + \frac{1}{\alpha_\theta}\right)}\|\hat{\theta}_k-\theta_k\|\nonumber\\
&\leq \left({L(K_1)+L(K_1,K_2)}+ \frac{1}{\alpha_\theta}\right)\|\hat{\theta}_k-\theta_k\|+\frac{1}{2\alpha_x}\|x_k-x_{k+1}\|\nonumber\\
&+\left\|\nabla_x f(x_{k+1},\theta_k)+\frac{1}{2\alpha_x}(x_{k+1}-x_k)+\nabla \phi(x_{k+1})\right\|,
\end{align}
being the second inequality a consequence of Assumption \ref{ass:1}(i), in particular of {the local Lipschitz continuity of $\nabla f$ in the bounded set $K_1\times K_2$ defined in \eqref{eq:K}}. From the optimality condition associated to the computation of the iterate $x_{k+1}$, see \eqref{eq:x_k+1}, we can write the following equality
\begin{equation}\label{eq:tech3}
    \nabla_x f(x_{k+1},\theta_k) = \frac{1}{\alpha_x}(y_k-x_{k+1})-\nabla\phi(x_k).
\end{equation}
We plug \eqref{eq:tech3} in \eqref{eq:tech2} and use again the triangular inequality, thus obtaining
\begin{align}\label{eq:tech4}
\|v_{k+1}\|&\leq \left({L(K_1)+L(K_1,K_2)}+\frac{1}{\alpha_\theta}\right)\|\hat{\theta}_k-\theta_k\|+\frac{1}{\alpha_x}(\|x_{k+1}-x_k\|+\|y_k-x_{k+1}\|)\nonumber\\
&+\|\nabla \phi(x_{k+1})-\nabla \phi(x_k)\|\nonumber\\
&\leq \frac{1}{\lambda_k}\left({L(K_1)+L(K_1,K_2)}+\frac{1}{\alpha_\theta}\right)\|\theta_{k+1}-\theta_k\|+\frac{2}{\alpha_x}\|x_{k+1}-x_k\|\nonumber\\
&+\frac{\rho}{\alpha_x}\|\nabla \phi(x_k)-\nabla \phi(x_{k-1})\|+\|\nabla \phi(x_{k+1})-\nabla \phi(x_k)\|\nonumber\\
&\leq \frac{1}{\lambda_k}\left({L(K_1)+L(K_1,K_2)}+\frac{1}{\alpha_\theta}\right)\|\theta_{k+1}-\theta_k\|+\left(\frac{2}{\alpha_x}+L_\phi\right)\|x_{k+1}-x_k\|\nonumber\\
&+\frac{\rho L_\phi}{\alpha_x}\|x_k-x_{k-1}\|
\end{align}
where the second inequality follows by {\sc Step  1.1}, {\sc Step 2.2}, {\sc Step 2.3} of Algorithm \ref{alg:1} and the triangular inequality, and the third is due to the Lipschitz continuity of $\nabla \phi$ (Assumption \ref{ass:1}(ii)).

Then, applying \eqref{eq:lambdamin} to \eqref{eq:tech4} allows us to obtain the following upper bound
\begin{align*}
\|v_{k+1}\|&\leq \frac{1}{\lambda_{\min}}\left({L(K_1)+L(K_1,K_2)}+\frac{1}{\alpha_\theta}\right)\|\theta_{k+1}-\theta_k\|+\left(\frac{2}{\alpha_x}+L_\phi\right)\|x_{k+1}-x_k\|\nonumber\\
&+\frac{\rho L_\phi}{\alpha_x}\|x_k-x_{k-1}\|\\
&\leq \left(\frac{1}{\lambda_{\min}}\left({L(K_1)+L(K_1,K_2)}+\frac{1}{\alpha_\theta}\right)+\left(\frac{2}{\alpha_x}+L_\phi\right)+\frac{\rho L_\phi}{\alpha_x}\right)\|\bar{z}_{k+1}-\bar{z}_k\|,
\end{align*}
where the second inequality follows by recalling the definition and block-structure of the sequence $\{\bar{z}_k\}_{k\in\N}$ in \eqref{eq:barzk}. Hence, by setting $q =  \frac{1}{\lambda_{\min}}\left({L(K_1)+L(K_1,K_2)}+\frac{1}{\alpha_\theta}\right)+\left(\frac{2}{\alpha_x}+L_\phi\right)+\frac{\rho L_\phi}{\alpha_x}$, the thesis follows.
\end{proof}

Next, we show that the value of the merit function \eqref{eq:merit} evaluated at $\hat{z}_{k+1}$ is controlled from above and below by the value of the same function at $\bar{z}_{k+1}$. 

\begin{lemma}\label{lem:H2}
Let Assumption \ref{ass:1} hold. Suppose that the sequence $\{(\xk,\theta_k)\}_{k\in\N}$ generated by Algorithm \ref{alg:1} is bounded. Let $\{\bar{z}_k\}_{k\in\N}$ and $\{\hat{z}_k\}_{k\in\N}$ be defined as in \eqref{eq:barzk} and \eqref{eq:tildezk}, respectively. Then, there exists a sequence of non-negative real numbers $\{t_k\}_{k\in\N}$ such that $\lim_{k\rightarrow \infty}t_k=0$ and
\begin{equation}\label{eq:H2}
H(\bar{z}_{k+1})\leq H(\hat{z}_{k+1})\leq H(\bar{z}_k)+t_k, \quad \forall \ k\in\N. 
\end{equation}
\end{lemma}

\begin{proof}
Starting from the value of $H(\bar{z}_{k+1})$, we immediately note that
\begin{align}
H(\bar{z}_{k+1})&=H(x_{k+1},x_k,\theta_{k+1})\nonumber\\
&=f(x_{k+1},\theta_{k+1})+\phi(x_{k+1})+\iota_{\Omega}(\theta_{k+1})+\frac{1}{4\alpha_x}\|x_{k+1}-x_k\|^2\nonumber\\
&\leq f(x_{k+1},\hat{\theta}_k)+\phi(x_{k+1})+\iota_{\Omega}(\hat{\theta}_{k})+\frac{1}{4\alpha_x}\|x_{k+1}-x_k\|^2\label{eq:H2_tech1}\\
&=H(x_{k+1},x_k,\hat{\theta}_k)=H(\hat{z}_{k+1}),\nonumber
\end{align}
where inequality \eqref{eq:H2_tech1} holds because of {\sc Step 2.3} of Algorithm \ref{alg:1} and the fact that both $\theta_{k+1}$ and $\hat{\theta}_k$ belong to $\Omega$. Thus, the left-hand inequality in \eqref{eq:H2} holds.

On the other hand, by using the block-Lipschitz continuity of $\nabla f$ stated in \eqref{eq:block_Lipschitz}, the upper bounds $L(K_1)$ and $L(K_2)$ given in \eqref{eq:Lbound} on the block-Lipschitz constants of $\nabla_x f(\cdot,\theta)$ and $\nabla_\theta f(x,\cdot)$ for all $(x,\theta)\in K_1\times K_2$, where $K_1\times K_2$ is defined as in \eqref{eq:K}, and the Lipschitz continuity of $\nabla \phi$, we have
\begin{align}
H(\hat{z}_{k+1})&=H(x_{k+1},x_k,\hat{\theta}_k)\nonumber\\
&=f(x_{k+1},\hat{\theta}_{k})+\phi(x_{k+1})+\iota_{\Omega}(\hat{\theta}_{k})+\frac{1}{4\alpha_x}\|x_{k+1}-x_k\|^2\nonumber\\
& \leq f(x_{k+1},\theta_k)+\nabla_\theta f(x_{k+1},\theta_k)^T(\hat{\theta}_k-\theta_k)+\frac{L(K_1)}{2}\|\hat{\theta}_k-\theta_k\|^2\nonumber\\
&+\phi(x_k)+\nabla \phi(x_k)^T(x_{k+1}-x_k)+\frac{L_\phi}{2}\|x_{k+1}-x_k\|^2+\iota_{\Omega}(\theta_{k})+\frac{1}{4\alpha_x}\|x_{k+1}-x_k\|^2\nonumber\\
&\leq f(x_{k+1},\theta_k)+\phi(x_k)+\iota_{\Omega}(\theta_k)+\frac{1}{2}\left(L_\phi+\frac{1}{2\alpha_x}\right)\|x_{k+1}-x_k\|^2\nonumber\\
&+\frac{L(K_1)}{2}\|\hat{\theta}_k-\theta_k\|^2+\nabla \phi(x_k)^T(x_{k+1}-x_k)\nonumber\\
&\leq f(x_k,\theta_k)+\nabla_x f(x_k,\theta_k)^T(x_{k+1}-x_k)+\frac{L(K_2)}{2}\|x_{k+1}-x_k\|^2+\phi(x_k)+\iota_{\Omega}(\theta_k)\nonumber\\
&+\frac{1}{2}\left(L_\phi+\frac{1}{2\alpha_x}\right)\|x_{k+1}-x_k\|^2+\frac{L(K_1)}{2}\|\hat{\theta}_k-\theta_k\|^2+\nabla \phi(x_k)^T(x_{k+1}-x_k)\nonumber\\
&\leq H(\bar{z}_k)+\frac{1}{2}\left(L_{\phi}+L(K_2)+\frac{1}{2\alpha_x}\right)\|x_{k+1}-x_k\|^2+\frac{L(K_1)}{2}\|\hat{\theta}_k-\theta_k\|^2\nonumber\\
&+(\|\nabla_x f(x_k,\theta_k)\|+\|\nabla \phi(x_k)\|)\|x_{k+1}-x_k\|.\label{eq:long_tech1}
\end{align}
In the above chain of inequalities, the first inequality is deduced by applying the descent lemma \cite[Theorem 18.15]{Bauschke-etal-2017} to {both the functions $f(x_{k+1},\cdot)$ and $\phi(\cdot)$}, and recalling that the block-Lipschitz constants of $\nabla_\theta f(x,\cdot)$ are bounded from above by $L(K_1)$ for all $x\in K_1$ (see \eqref{eq:Lbound}), being $K_1$ defined as in \eqref{eq:K}; the second inequality follows from \eqref{eq:d_k}; the third is obtained by applying again the descent lemma to the function $f(\cdot,\theta_k)$ and exploiting the bound $L(K_2)$ given in \eqref{eq:Lbound}, being $K_2$ defined as in \eqref{eq:K}; the fourth is trivially deduced by adding the term $\frac{1}{4\alpha_x}\|x_k-x_{k-1}\|^{{2}}$ and applying the Cauchy-Schwarz inequality.

Since $\nabla_x f(\cdot,\theta)$ is $L(K_2)-$Lipschitz continuous for all $\theta\in K_2$, $\nabla \phi$ is Lipschitz continuous, and the sequence $\{(x_k,\theta_k)\}_{k\in\N}$ is bounded, there exists a constant $\Lambda>0$ such that
\begin{equation*}
\|\nabla_x f(x_k,\theta_k)\|+\|\nabla \phi(x_k)\|\leq \Lambda.   
\end{equation*}
Hence, inequality \eqref{eq:long_tech1} entails
\begin{align*}
H(\hat{z}_{k+1})& \leq H(\bar{z}_k)+\frac{1}{2}\left(L_{\phi}+L(K_2)+\frac{1}{2\alpha_x}\right)\|x_{k+1}-x_k\|^2+\frac{L(K_1)}{2}\|\hat{\theta}_k-\theta_k\|^2\nonumber\\
&+\Lambda\|x_{k+1}-x_k\|\\
&\leq H(\bar{z}_k)+\frac{1}{2}\left(L_{\phi}+L(K_2)+\frac{1}{2\alpha_x}\right)\|x_{k+1}-x_k\|^2+\frac{L(K_1)}{2\lambda_k}\|\theta_{k+1}-\theta_k\|^2\nonumber\\
&+\Lambda\|x_{k+1}-x_k\|\\
&\leq H(\bar{z}_k)+\frac{1}{2}\left(L_{\phi}+\frac{L(K_1)}{\lambda_{\min}}+L(K_2)+\frac{1}{2\alpha_x}\right)\|\bar{z}_{k+1}-\bar{z}_k\|^2+\Lambda\|\bar{z}_{k+1}-\bar{z}_k\|,
\end{align*}
where the second inequality follows from {\sc Step 2.3} of Algorithm \ref{alg:1} {and $\lambda_k\leq 1$,} and the last inequality is due to \eqref{eq:lambdamin} and the definition of the sequence $\{\bar{z}_k\}_{k\in\N}$ given in \eqref{eq:barzk}.

Hence, the following relation has been obtained
\begin{equation*}
    H(\hat{z}_{k+1})\leq H(\bar{z}_k)+t_k,
\end{equation*}
where
\begin{equation}\label{eq:tk}
t_k = \frac{1}{2}\left(L_{\phi}+\frac{L(K_1)}{\lambda_{\min}}+L(K_2)+\frac{1}{2\alpha_x}\right)\|\bar{z}_{k+1}-\bar{z}_k\|^2+\Lambda\|\bar{z}_{k+1}-\bar{z}_k\|, \quad \forall \ k\in\N.    
\end{equation}
Finally, letting $H_{low}\in\R$ be any constant such that $H(x,y,\theta)\geq H_{low}$ for all $(x,y,\theta)\in\R^n\times\R^n\times \R^p$, which exists thanks to Assumption \ref{ass:1}(iii), and summing inequality \eqref{eq:sufficient_decrease} for $k=0,\ldots,K$, we come to
\begin{equation*}
    \sum_{k=0}^{K}\|\bar{z}_{k+1}-\bar{z}_k\|^2\leq \frac{1}{c}(H(\bar{z}_0)-H(\bar{z}_{K+1}))\leq \frac{1}{c}(H(\bar{z}_0)-H_{low}),
\end{equation*}
and as a result
\begin{equation}\label{eq:sumz}
\sum_{k=0}^{K}\|\bar{z}_{k+1}-\bar{z}_k\|^2<\infty,    
\end{equation}
which allows us to conclude that the sequence $\{t_k\}_{k\in\N}$ defined in \eqref{eq:tk} is converging to zero. Therefore, the right-hand inequality in \eqref{eq:H2} holds and the proof is complete.
\end{proof}

We are now ready to prove the convergence of Algorithm \ref{alg:1}. In order to obtain the desired result, we will assume that the merit function $H$ in \eqref{eq:merit} satisfies the Kurdyka--\L{}ojasiewicz inequality (KL) at any point $(x_*,y_*,\theta_*)\in\R^n\times \R^n\times \R^p$ \cite[Definition 3]{Bolte-etal-2014}, i.e., there exists $\upsilon \in (0,+\infty]$, a neighborhood $\mathcal{U}$ of  $(x_*,y_*,\theta_*)$ and a continuous concave function $\xi:[0,\upsilon)\longrightarrow [0,+\infty)$ with $\xi(0) = 0$, $\xi\in C^1(0,\upsilon)$, $\xi'(s) > 0$ for all $s \in (0,\upsilon)$, such that 
    \begin{equation}\label{eq:KL}
        \xi'(H(x,y,\theta)-H(x_*,y_*,\theta_*)) \operatorname{dist}(0,\partial H(x,y,\theta)) \geq 1,
    \end{equation}
    for all $(x,y,\theta)\in \mathcal{U}\cap \{(x,y,\theta)\in\R^n: \ F(x_*,y_*,\theta_*)<F(x,y,\theta)<F(x_*,y_*,\theta_*)+ \upsilon \}$.

\begin{theorem}
Let Assumption \ref{ass:1} hold. Suppose that the sequence $\{(\xk,\theta_k)\}_{k\in\N}$ generated by Algorithm \ref{alg:1} is bounded. If $\rho < \frac{1}{2L_\phi}$ and $\alpha_x<\frac{1-2L_\phi \rho}{2L_\phi}$, and the merit function \eqref{eq:merit} satisfies the KL inequality \eqref{eq:KL} on its domain, then the sequence $\{(\xk,\theta_k)\}_{k\in\N}$ converges to a stationary point of problem \eqref{eq:abstract_problem}.
\end{theorem}

\begin{proof}
Based on Lemmas \ref{lem:H1}-\ref{lem:H3}-\ref{lem:H2}, we have proved that there exist positive constants $c,q$ and a non-negative sequence $\{t_k\}_{k\in\N}$ converging to zero, such that for all $k\in\N$
\begin{align}
H(\bar{z}_{k+1})&\leq H(\bar{z}_{k}) - c\|\bar{z}_{k+1}-\bar{z}_k\|^2 \label{eq:fund1}\\
H(\bar{z}_{k+1})&\leq H(\hat{z}_{k+1})\leq H(\bar{z}_k)+t_k\label{eq:fund2}\\
\|v_{k+1}\|&\leq q\|\bar{z}_{k+1}-\bar{z}_k\|.\label{eq:fund3}
\end{align}
Furthermore, we have
\begin{align*}
\|\hat{z}_{k+1}-\bar{z}_k\|&\leq \|x_{k+1}-x_k\|+\|x_k-x_{k-1}\|+\|\hat{\theta}_{k}-\theta_k\|\\
&\leq \|x_{k+1}-x_k\|+\|x_k-x_{k-1}\|+\frac{1}{\lambda_k}\|\theta_{k+1}-\theta_k\|\\
&\leq \left(2+\frac{1}{\lambda_{\min}}\right)\|\bar{z}_{k+1}-\bar{z}_k\|,
\end{align*}
with the second inequality being a consequence of {\sc Step 2.3}, $\lambda_k\leq 1$ and \eqref{eq:lambdamin}, and the third inequality implied by the definition of $\{\bar{z}_k\}_{k\in\N}$ in \eqref{eq:barzk}. By property \eqref{eq:sumz}, the above inequality implies 
\begin{equation}\label{eq:fund4}
 \lim_{k\rightarrow \infty}\|\hat{z}_{k+1}-\bar{z}_k\|=0.
\end{equation}
Finally, by \eqref{eq:fund4} and the continuity of the merit function $H$ on its domain, it easily follows that
\begin{equation}\label{eq:fund5}
\forall  \ \{k_j\}_{j\in\N}\subseteq \N \ \text{such that} \ \lim_{j\rightarrow \infty}\bar{z}_{k_j}=z_* \quad \Rightarrow \quad \lim_{j\rightarrow \infty}H(\hat{z}_{k_j+1})=H(z_*).
\end{equation}
Properties \eqref{eq:fund1}-\eqref{eq:fund2}-\eqref{eq:fund3}-\eqref{eq:fund4}-\eqref{eq:fund5}, together with the boundedness of $\{(x_k,\theta_k)\}_{k\in\N}$ and the fact that $H$ satisfies the KL inequality, allow us to frame Algorithm \ref{alg:1} as a special instance of the abstract convergence framework formerly devised in \cite{Bonettini-Ochs-Prato-Rebegoldi-2023} and then simplified in \cite{Bonettini-Prato-Rebegoldi-2024}. In particular, by applying \cite[Theorem 14]{Bonettini-Prato-Rebegoldi-2024}, we obtain that the sequence $\{\bar{z}_k\}_{k\in\N}$ converges to a stationary point $(x_*,y_*,\theta_*)$ of $H$, i.e., $0\in\partial H(x_*,y_*,\theta_*)$. From the explicit expression of the subdifferential of $H$ given in \eqref{eq:merit_subdifferential}, we immediately get $x_*=y_*$, and thus
\begin{equation*}
    0\in\left(\begin{array}{c}
         \nabla_x f(x_*,\theta_*)+\nabla \phi(x_*)\\
          0\\
         \nabla_\theta f(x_*,\theta_*)+w_*
    \end{array}\right), \quad w_*\in\partial \iota_\Omega(\theta_*),
\end{equation*}
and based on the expression of the subdifferential of $F$ in \eqref{eq:subdifferential_F}, we conclude that $0 \in \partial F(x_*,\theta_*)$, so that $(x_*,\theta_*)$ is stationary for $F$.
\end{proof}

\begin{remark}
The merit function \eqref{eq:merit} satisfies the KL inequality \eqref{eq:KL} on its domain if the objective function $F$ is definable in an $o-$minimal structure \cite[Definition 6]{Bolte-etal-2007b}, as definable functions comply with \eqref{eq:KL} and finite sums of definable functions in some $o-$minimal structure remain in the same structure \cite[Remark 5]{Bolte-etal-2007b}. In the case of the blind model \eqref{eq:opt_prob_blind_super_res}, we can ensure $F$ is definable whenever the regularizer $\phi(x)=\frac{1}{2}\|x-N_\sigma(x)\|^2$ is definable in the same structure as that of the least squares term and the indicator function. This is the case when the employed neural network $N_\sigma$ is built upon definable functions such as ReLU, eLU, quadratics and SoftPlus functions, as the composition of definable functions remains definable \cite[Section 5.2]{Davis-etal-2020}.   
\end{remark}

\section{Numerical experiments}\label{sec:4}
In this section, we validate our approach on the super-resolution of real images acquired with a 1.5 T MRI scanner (Siemens) at the Neuroradiology Unit of Careggi University Hospital in Florence, Italy\footnote{The data were collected under the Italian PRIN 2022 PNRR Project “Advanced Optimization Methods for Automated Central Vein Sign Detection in Multiple Sclerosis from Magnetic Resonance Imaging (AMETISTA),” project code P2022J9SNP (CUP E53D23017980001 – B53D23027830001). The study protocol and data collection procedures were approved by the Ethics Committee “Area Vasta Centro", Tuscany Region, Italy. The data cannot be made available to individuals outside the project. }. In particular, we evaluate its performance on different slices extracted from both FLAIR and SWI sequences obtained at 1.5 T by visually comparing the super-resolved images with the corresponding slices from the 3 T acquisition. More specifically, we consider three different slices from the same patient, each exhibiting a visible lesion and distinct diagnostic features in both the FLAIR and SWI sequences. The corresponding 3 T and 1.5 T slices are shown in the first two columns of Figures \ref{fig:flair_lesion1} - \ref{fig:swi_lesion3}. Figures \ref{fig:flair_lesion1} and \ref{fig:swi_lesion1}, Figures \ref{fig:flair_lesion2} and \ref{fig:swi_lesion2}, Figures \ref{fig:flair_lesion3} and \ref{fig:swi_lesion3} report the FLAIR and the SWI sequences corresponding to the same brain slice.\\

\paragraph{\textit{Implementation details}} We first describe the implementation details of our proposal related to the objective function in \eqref{eq:opt_prob_blind_super_res}. The term $b$ corresponds to the native 1.5 T image of the considered sequence. Assuming periodic boundary conditions, the convolution $\theta * x$ is performed by means of the Fast Fourier Transform \cite{hansen2006deblurring} and the proximal operator associated to the term $f(\cdot,\theta)$ in \textsc{Step 1.2} of Algorithm \ref{alg:1} is computed using the closed form expression given in \cite{zhao2016fast}.

\noindent For the definition of $\phi$ in \eqref{eq:phi}, we select $N_{\sigma}$ as the differentiable version of the DRUNet architecture, originally proposed in \cite{zhang2021plug} and later adapted in \cite{hurault2022gradient}. More in detail, we employ the pretrained network weights released with the code of \cite{hurault2022gradient} and available through the the DeepInverse library \cite{tachella2025deepinverse}, fixing the hyper-parameter $\sigma=0.06$. 

\noindent Finally, the projection of the blurring kernel $\theta$ onto the set $\Omega$ in \textsc{Step 2.1} of Algorithm \ref{alg:1} is carried out applying the iterative algorithm proposed in \cite{Dai-etal-2006}. 

\noindent All the experiments have been performed on a workstation equipped with an Intel i9-285K processor and an NVIDIA RTX PRO 4000 GPU. The implementation is written in Python and it relies on the DeepInverse library \cite{tachella2025deepinverse}. \\

\paragraph{\textit{Model parameters}.} 
We heuristically selected the parameters $\lambda$ and $M$ in the definitions \eqref{eq:phi} and \eqref{eq:Omega}. In particular, we set $\lambda=0.15$ and $M=0.45$ for the FLAIR slices and $\lambda=0.075$ $M=0.6$ for the SWI data. The scale factor is set to $s = 2$. \\

\paragraph{\textit{Algorithm parameters}}  The initial iterates $x_{-1}$ and $x_0$ are initialized by bicubic interpolation of $b$ (with a proper shift correction as suggested in \cite{Zhang-et-al-2021}). The initial iterate $\theta_0$ is defined as a  square filter $13\times 13$ computed projecting onto the set $\Omega$ a Gaussian filter with standard deviation $1$.  Regarding the steplength parameters in \textsc{Step 1} of Algorithm \ref{alg:1}, they are fixed to $\rho=0.5$, $\alpha_x=1.34$. In \textsc{Step 2}, we set $\alpha_{\theta}=0.8$, while  the line search parameters related are chosen as $\gamma=0.5$ and $\nu=10^{-4}$. 
The maximum number of iterations is set to $100$ and it is combined with the following stopping criterion
\begin{equation}\label{eq:stop_Crit}
\frac{|f(x_{k+1},\theta_{k+1})+\phi(x_{k+1}) - f(x_{k},\theta_{k})-\phi(x_{k})|}{|f(x_{k},\theta_{k})+\phi(x_{k})|}\leq \varepsilon,
\end{equation}
with $\varepsilon = 10^{-5}$.\\ 

\paragraph{\textit{Numerical results.}} In the third column of Figures \ref{fig:flair_lesion1} - \ref{fig:swi_lesion3} we report the super resolved images obtained using our approach. Upon visual assessment, Figures \ref{fig:flair_lesion1}, \ref{fig:flair_lesion2} and \ref{fig:flair_lesion3} show that the super-resolved 1.5 T  FLAIR images exhibit a remarkable improvement in gray-white matter differentiation and sharpness of cortex, providing an easier identification and demarcation of white matter lesions. Overall, their quality closely resembles that of the 3 T FLAIR images.\\
In Figure \ref{fig:swi_lesion1}, \ref{fig:swi_lesion2} and \ref{fig:swi_lesion3}, the qualitative improvement was less pronounced on 1.5 T SWI sequences, where vascular structures are starkly less differentiated compared to 3 T SWI sequences. However, the super-resolved 1.5 T SWI showed a sharper demarcation of  small veins, allowing, in some cases, a more reliable classification of white matter lesions as CVS-positive or CVS-negative. 
Additional comments on the reconstructions obtained with our approach are provided in the figure captions to facilitate comparison between the discussion and the visual results.

\input{images/FLAIR_first}
\input{images/SWI_first}
\input{images/FLAIR_second}
\input{images/SWI_second}
\input{images/FLAIR_third}
\input{images/SWI_third}

\section{Conclusions}\label{sec:5}
In this work, we addressed the super-resolution of real 1.5 T MRI images acquired in clinical settings. The problem was formulated as a blind inverse problem involving the joint estimation of the high-resolution image and the blur kernel. The proposed approach integrates Plug-and-Play regularization for the image and suitable constraints for the kernel, and relies on a heterogeneous alternating block-coordinate scheme in which each subproblem is handled by a specifically designed algorithm. Convergence properties of the method have been rigorously established.
Experimental results on FLAIR and SWI sequences indicate improved structural sharpness and enhanced visibility of diagnostically relevant features, such as venous structures and white-matter lesions. Comparison with 3 T images from the same patient suggests that the method may partially bridge the gap between widely available low-field systems and higher-field scanners, without requiring hardware upgrades or supervised paired training data.
Future work will focus on validation on larger cohorts and extension to three-dimensional acquisitions.

\section*{Acknowledgments}
All the authors are partially supported by the  the Italian MUR through the PRIN 2022 PNRR Project “Advanced optimization METhods for automated central veIn Sign detection in multiple sclerosis from magneTic resonAnce imaging (AMETISTA)”, project code: P2022J9SNP (CUP  E53D23017980001 - B53D23027830001), under the National Recovery and Resilience Plan (PNRR), Italy, Mission 04 Component 2 Investment 1.1   funded by the European Commission - NextGeneration EU programme. \\
L.M. is supported by \#NEXTGENERATIONEU (NGEU) and funded by the Ministry
of University and Research (MUR), National Recovery and Resilience Plan (NRRP), project MNESYS (PE0000006) – A Multiscale integrated approach to the study of the nervous system in health and disease (DR. 1553 11.10.2022).\\
F.P., S. R. and A. S. are members of the Gruppo Nazionale per il Calcolo Scientifico (GNCS) of the Italian Istituto Nazionale di Alta Matematica (INdAM), which is kindly acknowledged.\\
\bibliographystyle{amsplain}
\bibliography{biblio}
\end{document}